\documentclass[reqno]{amsart}

\usepackage{amsmath}
\usepackage{amsfonts}
\usepackage{amssymb,enumerate}
\usepackage{amsthm}
\usepackage[all]{xy}
\usepackage{rotating}
\usepackage{hyperref}
\usepackage{color}

\theoremstyle{plain}
\newtheorem{lem}{Lemma}[section]
\newtheorem{cor}[lem]{Corollary}
\newtheorem{prop}[lem]{Proposition}
\newtheorem{thm}[lem]{Theorem}

\newtheorem{intthm}{Theorem}

\newtheorem*{mthm*}{Main Theorem}
\newtheorem*{mcor*}{Corollary}

\theoremstyle{definition}
\newtheorem{defn}[lem]{Definition}

\newtheorem{question}[lem]{Question}

\newtheorem{disc}[lem]{Remark}

\newtheorem{notn}[lem]{Notation}
\newtheorem{fact}[lem]{Fact}

\newtheorem{Proof}[lem]{Proof}

\newtheorem*{convention*}{Convention}

\newcommand{\Ht}{\operatorname{ht}}	
	
\newcommand{\depth}{\operatorname{depth}}	
\newcommand{\rank}{\operatorname{rank}}	

\newcommand{\edim}{\operatorname{edim}}

\newcommand{\soc}{\operatorname{Soc}}
\newcommand{\len}{\operatorname{length}}

\newcommand{\spec}{\operatorname{Spec}}

\newcommand{\im}{\operatorname{Im}}

\newcommand{\ideal}[1]{\mathfrak{#1}}
\newcommand{\m}{\ideal{m}}
\newcommand{\n}{\ideal{n}}
\newcommand{\p}{\ideal{p}}
\newcommand{\q}{\ideal{q}}
\newcommand{\fm}{\ideal{m}}

\newcommand{\fa}{\ideal{a}}
\newcommand{\fb}{\ideal{b}}
\newcommand{\fc}{\ideal{c}}

\newcommand{\comp}[1]{\widehat{#1}}

\newcommand{\wti}{\widetilde}

\newcommand{\Min}{\operatorname{Min}}

\newcommand{\xra}{\xrightarrow}

\newcommand{\onto}{\twoheadrightarrow}

\newcommand{\x}{\mathbf{x}}

\renewcommand{\geq}{\geqslant}
\renewcommand{\leq}{\leqslant}
\renewcommand{\ge}{\geqslant}
\renewcommand{\le}{\leqslant}

\newcommand{\Ext}[4][R]{\operatorname{Ext}_{#1}^{#2}(#3,#4)}

\newcommand{\Hom}{\operatorname{Hom}}

\newcommand{\ssm}{\smallsetminus}

\def\Ext{\operatorname{Ext}}

\newcommand{\ecodepth}{\operatorname{ecodepth}}

\newcommand{\EXT}[4][R]{\operatorname{Ext}_{#1}^{#2}(#3,#4)}

\newcommand{\HOM}[3][R]{\operatorname{Hom}_{#1}(#2,#3)}

\newcommand{\Nil}{\operatorname{Nil}}
\newcommand{\reg}{\operatorname{reg}}

\newcommand{\ulx}{\underline{X}}
\newcommand{\uly}{\underline{Y}}
\newcommand{\pseries}[2]{#1[\![#2]\!]}
\newcommand{\psy}[1][k]{\pseries{#1}{\uly}}
\newcommand{\psx}[1][k]{\pseries{#1}{\ulx}}
\newcommand{\psyx}[1][k]{\pseries{#1}{\uly,\ulx}}
\newcommand{\pyy}{k(\!(\underline{Y})\!)}

\numberwithin{equation}{lem}

\begin{document}

\bibliographystyle{amsplain}

\title[Decomposable maximal ideals]{Applications and homological properties of local rings with decomposable maximal ideals}

\author[S. Nasseh]{Saeed Nasseh}

\address{S. Nasseh, Department of Mathematical Sciences,
Georgia Southern University,
Statesboro, Georgia 30460, USA}

\email{snasseh@georgiasouthern.edu}
\urladdr{https://cosm.georgiasouthern.edu/math/saeed.nasseh}

\author[S. Sather-Wagstaff]{Sean Sather-Wagstaff}

\address{S. Sather-Wagstaff, Department of Mathematical Sciences,
Clemson University,
O-110 Martin Hall, Box 340975, Clemson, S.C. 29634
USA}

\email{ssather@clemson.edu}
\urladdr{https://ssather.people.clemson.edu/}

\author[R. Takahashi]{Ryo Takahashi}
\address{R. Takahashi, Graduate School of Mathematics\\
Nagoya University\\
Furocho, Chikusaku, Nagoya, Aichi 464-8602, Japan}
\email{takahashi@math.nagoya-u.ac.jp}
\urladdr{http://www.math.nagoya-u.ac.jp/~takahashi/}

\author[K. VandeBogert]{Keller VandeBogert}

\address{K. VandeBogert, Department of Mathematics,
University of South Carolina,
1523 Greene Street, Columbia, SC 29208 USA}

\email{kellerlv@math.sc.edu}

\thanks{Takahashi was partly supported by JSPS Grants-in-Aid
for Scientific Research 16K05098.}




\begin{abstract}
We construct a local Cohen-Macaulay ring $R$ with a prime ideal $\mathfrak{p}\in\spec(R)$ such that $R$ satisfies the uniform Auslander condition (UAC), but the localization $R_{\mathfrak{p}}$ does not  satisfy Auslander's condition (AC). Given any positive integer $n$, we also construct a local Cohen-Macaulay ring $R$ with a prime ideal $\mathfrak{p}\in\spec(R)$ such that $R$ has exactly two non-isomorphic semidualizing modules, but the localization $R_{\mathfrak{p}}$ has $2^n$ non-isomorphic semidualizing modules. Each of these examples is constructed as a fiber product of two local rings over their common residue field.  Additionally, we characterize the non-trivial Cohen-Macaulay fiber products  of finite Cohen-Macaulay type.
\end{abstract}

\maketitle


\section{Introduction}\label{sec161010a}

Throughout this paper, let $(R,\m_R,k)$ be a (commutative noetherian) local ring.

\

Many invariants and properties of local rings  behave well with respect to localization.
For instance,  if 
$\p\in\spec(R)$, then $\dim(R_\p)\leq\dim(R)$; if moreover
$R$ is Cohen-Macaulay, then
$R_{\p}$ is Cohen-Macaulay
and the type of $R_\p$ is at most the type of $R$. 
On the other hand, some invariants
and properties can behave unpredictably, e.g., the depth of $R$ can increase, decrease, or not change under localization. 

The first point of this article is to construct two examples of local rings with interesting localization behavior with respect to homological 
properties, beginning with the following.

\subsection*{Auslander Conditions}
Section~\ref{sec171226a} is devoted to the following 
conditions
from~\cite{huneke:svegr}, motivated by a conjecture of Auslander 
from~\cite[p.~815]{MR1674397}
and~\cite[p.~1]{happel:hcrt}.
\begin{enumerate}[(U{A}C)]
\item[(AC)]
for every finitely generated 
$R$-module
$M$ there exists an integer
$b_M\geq 0$ such that, for every finitely generated 
$R$-module
$N$, if 
$\EXT iMN=0$ for  $i\gg 0$, then 
$\EXT iMN=0$ for all $i>b_M$.
\item[(UAC)]
there exists an integer
$b\geq 0$ such that,
for all finitely generated 
$R$-modules
$M$ and $N$, if 
$\EXT iMN=0$ for  $i\gg 0$, then 
$\EXT iMN=0$ for all $i>b$.
\end{enumerate}

The following theorem (proved in Proof~\ref{para180107a})
is the first main result of this paper.
It says that the conditions (AC) and (UAC) need not be preserved by localization.
One interpretation of the result is that these conditions
are poor measures of the severity of the singularity of $R$, even when $R$ is complete, since intuition says that
$R_\p$ should  be no more singular than $R$.
See also Theorem~\ref{thm180128a}.

\begin{intthm}\label{thm171226a}
Let $k$ be any field which is not algebraic over a finite field. 
There exists a complete local Cohen-Macaulay $k$-algebra $R$ of Krull dimension 1
with decomposable maximal ideal
such that 
\begin{enumerate}[\rm(a)]
\item
$R$ has type~2, 
multiplicity 13, and embedding codepth 6
and satisfies (UAC),
\item
$R$ is the completion of a local ring that is essentially of finite type over $k$, 
and
\item 
there is a prime ideal $\p\in\spec(R)$ such that the localization
$R_{\p}$ is Gorenstein of multiplicity 12 
and embedding codepth 5
and does not even satisfy (AC).
\end{enumerate}
\end{intthm}

\subsection*{Semidualizing Modules}
Section~\ref{sec170206a} focuses on the following modules introduced by
Foxby~\cite{foxby:gmarm}. 
A finitely generated $R$-module $C$ is \emph{semidualizing} if
one has $R\cong \HOM CC$ and $\EXT iCC=0$ for all $i\geq 1$.
For instance, the free module $R^1$ is semidualizing.
A \emph{dualizing module} for $R$ is a semidualizing $R$-module of finite injective dimension;
this is the same as Grothendieck's \emph{canonical module} over a Cohen-Macaulay ring from~\cite{hartshorne:lc}.

The next result is proved in Proof~\ref{para180107b}.
It say that the number of isomorphism classes of semidualizing modules can grow under localization.
As with the preceding result, one can interpret this as saying that the number of isomorphism classes of semidualizing $R$-modules
is a poor measure of the severity of the singularity of $R$.

\begin{intthm}\label{thm170508a}
Let $n$ be a positive integer, and let $k$ be a field.
Then there is a 
complete local Cohen-Macaulay $k$-algebra $R$ of Krull dimension 1 with decomposable maximal ideal
such that
\begin{enumerate}[\rm(1)]
\item \label{thm170508a0} 
$R$ is the completion of a local ring that is essentially of finite type over $k$,
\item\label{thm170508a1} $R$ has type $1+2^n$, multiplicity $1+3^n$, 
embedding codepth $2n+1$,
and 
only two non-isomorphic semidualizing modules, namely its free module of rank 1 and its dualizing module, and
\item\label{thm170508a2} $R$ has a prime ideal $\p$ such that  
$R_\p$
has 
type $2^n$, multiplicity $3^n$, 
embedding codepth $2n$,
and exactly
$2^n$ non-isomorphic semidualizing modules.
\end{enumerate}
\end{intthm}

\subsection*{Finite Cohen-Macaulay Type}
The rings of Theorems~\ref{thm171226a} and~\ref{thm170508a}  have decomposable maximal ideals,
which means that they are realized as non-trivial fiber products; see Fact~\ref{disc161010a} below. 
Such rings have very rigid homological properties, as one sees, e.g., 
in~\cite{christensen:gmirlr, kostrikin, lescot:sbpfal, nasseh:vetfp, ogoma:edc}.
For instance, they satisfy (UAC), they have at most the two trivial semidualizing modules, 
they are G-regular, and they satisfy the Auslander-Reiten Conjecture. 

On the other hand, some properties are rarely satisfied by rings with decomposable maximal ideals.
For instance, they are never integral domains (hence, never regular), and they are rarely Gorenstein;
see~\cite[Observation~3.2]{christensen:gmirlr} or Corollary~\ref{cor161010b} below where it is shown that
such a ring is Gorenstein if and only if it is a hypersurface of dimension 1. 
A related characterization of the Cohen-Macaulay property is in Fact~\ref{para170508b}.

This leads to the second point of the present article: to 
similarly characterize when a local ring $R$ with decomposable maximal ideal is Cohen-Macaulay with
\emph{finite Cohen-Macaulay type}, that is, with only finitely many isomorphism classes of indecomposable
maximal Cohen-Macaulay modules.
This is the subject of Section~\ref{sec171118c}, the main result of which is stated next
and proved in Proof~\ref{para180107c} below.
For it, recall that a local ring is \emph{analytically unramified} provided that its completion is reduced.
In particular,  
each analytically unramified local ring is itself reduced.

\begin{intthm} \label{thm171221a}
Assume that the maximal ideal $\m_R$ is decomposable.
Then the following conditions are equivalent.
\begin{enumerate}[\rm(i)]
\item\label{thm171221a1}
The ring $R$ is Cohen-Macaulay and has finite Cohen-Macaulay type. 
\item\label{thm171221a2}
One has $R\cong S\times_k T$ where
$S$ is an analytically unramified hypersurface of Krull dimension 1 and multiplicity at most 2
and $T$ is regular of Krull dimension 1, i.e., a discrete valuation ring.
\item\label{thm171221a3}
One has $R\cong S\times_k T$ where
$S$ and $T$ are  Cohen-Macaulay of Krull dimension~1 with finite Cohen-Macaulay type and $R$ has multiplicity at most 3.
\end{enumerate}
When these conditions are satisfied, one has $\dim(R)=1$.
\end{intthm}

In addition, Theorem~\ref{thm171229a} characterizes the local rings of depth at most 1 with decomposable maximal ideal
and finite Cohen-Macaulay type.
Lastly, we note that Section~\ref{sec171120a} contains foundational material.

\section{Background on Local Rings}\label{sec171120a}

We begin this section with an observation of 
Ogoma~\cite[Lemma 3.1]{ogoma:edc}.

\begin{fact}\label{disc161010a}
A local ring $R$ is a \emph{fiber product} if it has the form
$$
R\cong S\times_k T=\left\{(s,t)\in S\times T\mid \pi_S(s)=\pi_T(t)\right\}
$$
where 
$(S,\fm_S,k)$ and $(T,\fm_T,k)$
are local rings with a common residue field $k$, and
$\pi_S\colon S\to k$ and $\pi_T\colon T\to k$ are the natural surjections.
Such a fiber product is \emph{non-trivial} if $S\neq k\neq T$.
Note that $S\times_k T$ is automatically a local ring with maximal ideal $\fm_{S\times_k T}=\fm_S\oplus \fm_T$ and residue field $k$. 
In particular, the maximal ideal of $S\times_k T$ is decomposable. 
We identify the maximal ideals
of $S$ and $T$  with the ideals 
$\m_S\oplus 0$ and $0\oplus\m_T$, 
respectively,
of $S\times_k T$.
Under this identification,  there are ring isomorphisms $S\cong (S\times_k T)/\frak m_T$ and 
$T\cong (S\times_k T)/\frak m_S$.

Conversely, assume that $\m_R$ is decomposable as $\m_R= I\oplus J$, and set $S=R/I$ and $T=R/J$.
Then there is an isomorphism
$R\xra\cong S\times_{k}T$ given by
$r\mapsto (r+I,r+J)$.
Under this isomorphism, we have $I\cong \m_T$ and $J\cong \m_S$.
\end{fact}

\begin{fact}\label{para170508b}
Assume  
that $\m_R$ is decomposable as $\m_R= I\oplus J$, and set $S=R/I$ and $T=R/J$.
By work of Lescot~\cite{lescot:sbpfal}, we have equalities
\begin{align}
\dim R&=\max\{\dim S,\dim T\}\notag \\
\depth R&=\min\{\depth S,\depth T,1\}.\notag
\end{align}
In particular, $R$ is Cohen-Macaulay if and only if 
$S$ and $T$ are both Cohen-Macaulay and $\dim S=\dim T\leq 1$.
See also Christensen, Striuli, and Veliche~\cite[Remark 3.1]{christensen:gmirlr}.
\end{fact}

\begin{defn}
Recall that the \emph{Poincar\'e series} and the \emph{Bass series} of 
a finitely generated $R$-module $M$
are 
the formal power series
\begin{align*}
P_M^R(t)&:=\sum_{i\geq 0}\rank_{k}(\Ext_R^i(M,k))t^i\\
I^M_R(t)&:=\sum_{i\geq 0}\rank_{k}(\Ext_R^i(k,M))t^i
\end{align*}
respectively.
The constant term of $P^R_M(t)$ is the minimal number of generators of $M$,
denoted $\beta^R_0(M)=\rank_{k}(\Hom_R(M,k))$.
The order of the Bass series $I_R^R(t)$ is $\depth(R)$, and the 
coefficient of $t^{\depth R}$ in $I_R^R(t)$ is the \emph{type of $R$}, denoted $r(R)$. 
\end{defn}

\begin{defn}\label{defn171118a}
Let $\comp R$ denote the $\m_R$-adic completion of $R$.
Then $R$ is a \emph{hypersurface} if there is a presentation $\comp R\cong Q/\fa$  where $Q$ is a regular local ring
and $\fa$ is principal.
Equivalently, if the \emph{embedding codepth} of $R$ is $\ecodepth(R)=\beta^R_0(\m_R)-\depth(R)$,
then $R$ is a hypersurface if and only if $\ecodepth(R)\leq 1$.
(In particular, we take the perspective that a regular local ring is a hypersurface.)
The \emph{embedding dimension} of $R$ is $\edim(R)=\beta^R_0(\m_R)$.
\end{defn}

\begin{disc}\label{disc171118c}
An alternate definition of the term ``hypersurface''  asks $R$ itself 
to be of the form $Q/\fa$ where $Q$ is regular and $\fa$ is principal. 
Heitmann and Jorgensen~\cite{MR2975397} show that our definition is less restrictive than this one.
\end{disc}

Next, we document the
Bass series and type of a non-trivial fiber product.

\begin{fact}\label{para170508bz}
Assume  that $\m_R$ is decomposable as $\m_R= I\oplus J$, and set $S=R/I$ and $T=R/J$.
Using results  of Lescot~\cite[Theorem 3.1]{lescot:sbpfal} and Kostrikin and {\v{S}}afarevi{\v{c}}~\cite{kostrikin},
the Bass series and type of $R$ are as follows.

Case 1: $S$ and $T$ are both singular, i.e., non-regular.
\begin{align*}
I_R^R(t)
&=\frac{tP^S_{k}(t)P^T_{k}(t)+I_S^S(t)P^T_{k}(t)+I_T^T(t)P^S_{k}(t)}{P^T_{k}(t)+P^S_{k}(t)-P^S_{k}(t)P^T_{k}(t)}
\\
r(R)
&=
\begin{cases}
r(S)+r(T)&\text{if $\depth S=0=\depth T$}
\\
r(S)&\text{if $\depth S=0<\depth T$}
\\
r(T)&\text{if $\depth T=0<\depth S$}
\\
r(S)+r(T)+1&\text{if $\depth S=1=\depth T$}
\\
r(S)+1&\text{if $\depth S=1<\depth T$}
\\
r(T)+1&\text{if $\depth T=1<\depth S$}
\\
1&\text{if $1<\depth S,\depth T$.}
\end{cases}
\end{align*}
In particular, recalling the characterization of the Cohen-Macaulay property for $R$ from Fact~\ref{para170508bz},
we conclude
in this case that $R$ is not Gorenstein.
(This follows from a case-by-case analysis of the options for $r(R)$.
For instance, if $\depth S=0=\depth T$, then $R$ is Cohen-Macaulay but has 
$r(R)
=r(S)+r(T)\geq 2$.
In the next case $\depth S=0<\depth T$, then $R$ is not Cohen-Macaulay.
And so on.)

Case 2: $S$ is 
singular and $T$ is regular of dimension $n\geq 1$.
(We require $n\geq 1$ in this case, otherwise $T=k$, contradicting the non-triviality of our fiber product.)
\begin{align*}
I_R^R(t)
&=\frac{tP^S_{k}(t)(1+t)^n+I_S^S(t)(1+t)^n-t^{n+1}P^S_{k}(t)}{P^T_{k}(t)+P^S_{k}(t)-P^S_{k}(t)P^T_{k}(t)}
\\
r(R)
&=\begin{cases}
r(S)&\text{if $\depth S=0$}
\\
r(S)+1&\text{if $\depth S=1$}
\\
1&\text{if $1<\depth S$.}
\end{cases}
\end{align*}
Also as in the previous case, we conclude that $R$ is never Gorenstein in this case.

Case 3: $S$ and $T$ are regular of dimensions $m, n\geq 1$ respectively.
\begin{align*}
I_R^R(t)
&=\frac{t(1+t)^{m+n}-t^{m+1}(1+t)^{n}-t^{n+1}(1+t)^{m}}{P^T_{k}(t)+P^S_{k}(t)-P^S_{k}(t)P^T_{k}(t)}
\\
r(R)
&=1
\end{align*}
Note  in this 
case
that $R$ is Gorenstein if and only if it is Cohen-Macaulay, i.e.,
if and only if $m=n=1$.
See also Christensen, Striuli, and Veliche~\cite[Remark 3.1]{christensen:gmirlr}.
\end{fact}

The next result summarizes some of the points of Fact~\ref{para170508bz}.
It is almost certainly well-known to the experts, in light of~\cite{christensen:gmirlr,lescot:sbpfal},

\begin{cor}\label{cor161010b}
The following conditions on the local ring $R$ are equivalent.
\begin{enumerate}[\rm(i)]
\item\label{cor161010b2}
The maximal ideal $\m_R$ is decomposable, and $R$ is Gorenstein.
\item\label{cor161010b3}
The maximal ideal $\m_R$ is decomposable, and $R$ is a dimension-1 hypersurface.
\item\label{cor161010b4}
There is an isomorphism $R\cong S\times_k T$ where $(S,\m_S,k)$ and $(T,\m_T,k)$ are regular of dimension 1, i.e., 
where $S$ and $T$
are discrete valuation rings.
\end{enumerate}
\end{cor}

\begin{proof}
For $\eqref{cor161010b4}\implies\eqref{cor161010b3}$, see~\cite[Observation~3.2]{christensen:gmirlr}.
The implication $\eqref{cor161010b3}\implies\eqref{cor161010b2}$ is old news, and
$\eqref{cor161010b2}\implies\eqref{cor161010b4}$
is from Fact~\ref{para170508bz}.
\end{proof}

\begin{defn}
The \emph{(Hilbert-Samuel) multiplicity} of the local ring $R$ is the positive integer
$e(R)=\lim_{n\to\infty}
\len(R/\m_R^{n+1})/n^d$.
In other words, $e(R)$ is the 
normalized
leading coefficient of the Hilbert polynomial representing the Hilbert function
$\len(R/\m_R^{n+1})$ for $n\gg 0$.
Recall the formula
\begin{equation}\label{eq180119a}
e(R)=\sum_{\substack{\dim(R/\p)\\=\dim(R)}}\len(R_{\p})e(R/\p)
\end{equation}
where the sum is taken over all primes $\p\in\spec(R)$ with $\dim(R/\p)=\dim(R)$;
see, e.g., \cite[A.26~Proposition]{MR2919145}.
\end{defn}

\begin{fact}\label{fact171221a}
Let $R\cong S\times_k T$ be a non-trivial fiber product.
The condition $\m_S\m_T=0$ 
implies
that every prime ideal of $R$
contains $\m_S$ or $\m_T$. From the isomorphisms $R/\m_S\cong T$ and $R/\m_T\cong S$, we
then conclude that 
$$\spec(R)=\{\m_S\oplus\q\mid\q\in\spec(T)\}\cup \{\p\oplus\m_T\mid\p\in\spec(S)\}.$$
Following this notation, we have
$\Ht(\m_S\oplus\q)=\Ht(\q)$ and 
$R/(\m_S\oplus\q)\cong T/\q$, so
$\dim(R/(\m_S\oplus\q))=\dim(T/\q)$, and similarly for the prime $\p\oplus\m_T$.
Furthermore, it is straightforward to show that if $\q\neq\m_T$, then
$R_{\m_S\oplus\q}\cong T_\q$, and similarly for $\p\oplus\m_T$.
Thus, we use the  
formula~\eqref{eq180119a}
to deduce the first three cases of the following formula:
\begin{align*}
e(R)
&=\begin{cases}
e(S)&\text{if $\dim(S)>\dim(T)$}\\
e(T)&\text{if $\dim(S)<\dim(T)$}\\
e(S)+e(T)&\text{if $\dim(S)=\dim(T)\neq 0$}\\
e(S)+e(T)-1&\text{if $\dim(S)=\dim(T)= 0$.}
\end{cases}
\end{align*}
For the fourth case, assume that $\dim(S)=\dim(T)= 0$. 
The isomorphisms $R/\m_T\cong S$ and $T/\m_T\cong k$ explain the second and third equalities in the next sequence.
\begin{align*}
e(R)&=\len(R)=\len(S)+\len(\m_T)\\
&=\len(S)+\len(T)-1=e(S)+e(T)-1
\end{align*}
The other equalities come from the dimension assumption for $S$ and $T$.
\end{fact}

\begin{fact}\label{lem171124a}
Let $\ulx,\uly$ be finite lists of power series variables over the field $k$, and let $\pyy$ denote the field of fractions of $\psy$.
Let $I\subseteq (\ulx)\psyx$ be an ideal of $\psyx$ such that
each $X_i$ has a positive power in $I$, that is, such that the quotient $\psyx_{(\ulx)}/(I)$ has finite, positive length. 
Then it is straightforward to show that there is an isomorphism 
$$\frac{\psyx_{(\ulx)}}{(I)}\cong\frac{\psx[\pyy]}{(I)}.$$
\end{fact}

\section{Auslander Conditions and Localizations}
\label{sec171226a}

With the preliminaries out of the way, we dive into the proof of our first main theorem from the introduction.

\begin{Proof}[\protect{Proof of Theorem~\ref{thm171226a}}]\label{para180107a}
A result of Jorgensen and \c Sega~\cite[Theorem~(1)]{jorgensen:nccgr}
provides a  finite dimensional Gorenstein $k$-algebra $A$ that does not satisfy (AC). 
Moreover, the ring $A$ is constructed as follows.
Fix an element $\alpha\in k$ that has infinite order in the multiplicative group $k^\times$.
In the polynomial ring $k[X_1,\ldots,X_5]$, consider the ideal $\fa$ generated by the following quadratic forms:
\begin{gather*}
\alpha X_1X_3+X_2X_3,
\
X_1X_4+X_2X_4,
\
X_3^2+\alpha X_1X_5-X_2X_5,
\\
X_4^2+X_1X_5-X_2X_5,
\
X_1^2,
\
X_2^2,
\
X_3X_4,
\
X_3X_5,
\
X_4X_5,
\
X_5^2.
\end{gather*}
Let $\fb=\fa k[\![X_1,\ldots,X_5]\!]$.
Then $A=k[X_1,\ldots,X_5]/\fa\cong k[\![X_1,\ldots,X_5]\!]/\fb$ is a  finite dimensional Gorenstein $k$-algebra
of length 12 
and embedding dimension 5
that does not satisfy (AC). 
(The isomorphism here is due to the fact that $\fa$ contains a power of each 
variable
$X_i$.)
In particular, if $\alpha$ is fixed, and $Y$ is a power series variable,
let
$k(\!(Y)\!)$ denote the field of fractions of the power series ring $k[\![Y]\!]$ as in Fact~\ref{lem171124a}
and set $\fc=\fa k(\!(Y)\!)[\![X_1,\ldots,X_5]\!]$;
then the ring $k(\!(Y)\!)[\![X_1,\ldots,X_5]\!]/\fc$
is a  finite dimensional Gorenstein $k(\!(Y)\!)$-algebra of length 12 
and embedding dimension~5
that does not satisfy~(AC).

Now, set $\fa'=\fa k[\![Y,X_1,\ldots,X_5]\!]$ and
$S=k[\![Y,X_1,\ldots,X_5]\!]/\fa'\cong A[\![Y]\!]$.
Then $S$ is local and Gorenstein of Krull dimension 1 and multiplicity 12.
Let $Z$ be another power series variable,
and consider the fiber product
\begin{equation}\label{eq171226a}
R=S\times_kk[\![Z]\!]\cong \frac{k[\![Y,X_1,\ldots,X_5,Z]\!]}{(\fa,YZ,X_1Z,\ldots,X_5Z)}
\end{equation}
The isomorphism comes from the fact that the quotient ring in this display has decomposable maximal ideal
$(Z)\oplus(X_1,\ldots,X_5,Y)$, by construction; see Fact~\ref{disc161010a}.
Then $R$ is a complete local $k$-algebra.
(One can also deduce the 
completeness
of $R$ from~\cite[Lemme(19.3.2.1(iii))]{grothendieck:ega4-4}.)
Also, $R$ is the completion of a local ring that is essentially of finite type over $k$, 
because of~\eqref{eq171226a} and the specific list of generators for $\fa$.
From~\cite[Corollaries~6.3 and~6.6]{nasseh:lrqdmi}, we know that $R$ satisfies (UAC) 
with $b=\depth(R)$.
Furthermore, $R$ is
Cohen-Macaulay of Krull dimension 1, type 2, and multiplicity 13,
by Facts~\ref{para170508b}, \ref{para170508bz}, and~\ref{fact171221a}.
By inspecting~\eqref{eq171226a}, one concludes that $\edim(R)=7$.

Set 
$P=(X_1,\ldots,X_5,Z)k[\![Y,X_1,\ldots,X_5,Z]\!]$ and
exploit the isomorphism~\eqref{eq171226a} to set
$\p=PR$. 
One checks readily that $\p$ is prime.
We localize at $\p$ in the next display, using  the condition $Y\notin P$ for the second isomorphism.
\begin{align*}
R_{\p}
&\cong\frac{k[\![Y,X_1,\ldots,X_5,Z]\!]_{P}}{(\fa,YZ,X_1Z,\ldots,X_5Z)}
\cong\frac{k[\![Y,X_1,\ldots,X_5]\!]_{(X_1,\ldots,X_5)}}{(\fa)}
\cong\frac{k(\!(Y)\!)[\![X_1,\ldots,X_5]\!]}{\fc}
\end{align*}
The third isomorphism is from Fact~\ref{lem171124a}.
As we noted in the first paragraph of this proof, this ring is Gorenstein of multiplicity 12 
and embedding dimension 5
and
does not satisfy (AC), so $R$ 
and $R_\p$ have
the desired properties.\qed
\end{Proof}

\begin{disc}
Of course, we would love to have an example like the above one with type 1, 
since that would answer the 
question~\cite[Section~6, \#3]{huneke:svegr}
of whether the AB property 
localizes.
(Recall that $R$ is an \emph{AB ring} if it is Gorenstein and satisfies (UAC).)
However, since Gorenstein rings with decomposable maximal ideals are hypersurfaces by Corollary~\ref{cor161010b},
and hypersurfaces are AB rings by~\cite[Theorem~4.7]{avramov:svcci}, our methods here will not yield such an example.
\end{disc}

Proof~\ref{para180107a} above uses the ring from~\cite[Theorem~(1)]{jorgensen:nccgr}
that is denoted $A$ in that paper. 
Using the ring $B$ from~\cite[Corollary~3.3(2)]{jorgensen:nccgr} in the same way, 
we obtain the following result.

\begin{thm}\label{thm180128a}
Let $k$ be any field which is not algebraic over a finite field. 
There exists a complete local Cohen-Macaulay $k$-algebra $R$ of Krull dimension 1
with decomposable maximal ideal
such that 
\begin{enumerate}[\rm(a)]
\item
$R$ has type~4, 
multiplicity 9, and embedding codepth 5
and satisfies (UAC),
\item
$R$ is the completion of a local ring that is essentially of finite type over $k$, 
and
\item 
there is a prime ideal $\p\in\spec(R)$ such that the localization
$R_{\p}$ has type 3, multiplicity 8, 
and embedding codepth 4
and does not even satisfy (AC).
\end{enumerate}
\end{thm}

\begin{disc}\label{disc180128a}
We do not know whether the rings constructed in Theorems~\ref{thm171226a} and~\ref{thm180128a} are minimal with respect to 
embedding codepth, type, or multiplicity. 
It is known from~\cite{avramov:phcnr} that if 
$\ecodepth(R_\p)\leq 3$ or if
$R_\p$ is Gorenstein with $\ecodepth(R_\p)\leq 4$
or if $e(R_\p)\leq 7$, then $R_\p$ satisfies (UAC).
However, it is not clear what happens, say, when $r(R_\p)=2$ or $e(R)=8$. Thus, we pose the following.
\end{disc}

\begin{question}\label{q180128a}
Does there exist a complete local Cohen-Macaulay ring $R$ of Krull dimension 1
with a prime ideal $\p\in\spec(R)$
such that the following hold?
\begin{enumerate}[\rm(a)]
\item
$R$ has type~2, 
multiplicity 8, and embedding codepth 5
and satisfies (UAC),~and
\item 
the localization
$R_{\p}$ is Gorenstein of multiplicity 8
and embedding codepth 5
and does not  satisfy (AC).
\end{enumerate}
or
\begin{enumerate}[\rm(a)]
\item
$R$ has type~2, 
multiplicity 8, and embedding codepth 4
and satisfies (UAC),~and
\item 
the localization
$R_{\p}$ has type 2, multiplicity 8, 
and embedding codepth 4
and does not  satisfy (AC).
\end{enumerate}
\end{question}

\section{Semidualizing Modules and Localization} \label{sec170206a}

Next, we prove our second main theorem from the introduction.

\begin{Proof}[\protect{Proof of Theorem~\ref{thm170508a}}]\label{para180107b}
To build the ring $R$, 
consider the list of variables $\ulx=X_{1,1},X_{2,1},\ldots,X_{1,n},X_{2,n}$, and
set 
$$I=(X_{1,1},X_{2,1})^2+\cdots+(X_{1,n},X_{2,n})^2\subseteq k[\![\ulx]\!].$$
Then the quotient ring
$$S_0\cong k[\![\ulx]\!]/I\cong (k[X_{1,1},X_{2,1}]/(X_{1,1},X_{2,1})^2)\otimes_k\cdots\otimes_k(k[X_{1,n},X_{2,n}]/(X_{1,n},X_{2,n})^2).$$
is artinian and local of type $2^n$, 
embedding dimension $2n$,
and length $3^n$.

Let $Y$ be another variable, and set 
$$S=k[\![\ulx,Y]\!]/(I)\cong S_0[\![Y]\!].$$
Note that $S$ is Cohen-Macaulay of dimension 1, type $2^n$, 
embedding dimension $2n+1$,
and multiplicity  $3^n$.

Also, let $T=k[\![Z]\!]$ where $Z$ is yet another variable, and set 
\begin{equation}\label{eq170508z}
R=S\times_kT\cong \frac{k[\![Z,\underline X,Y]\!]}{(I,Z\ulx,ZY)}.
\end{equation}
The isomorphism comes from the fact that the quotient ring in this display has decomposable maximal ideal
$(Z)\oplus(\ulx,Y)$, by construction; see Fact~\ref{disc161010a}.
Then $R$ is Cohen-Macaulay of dimension 1, type $1+2^n$, and multiplicity $1+3^n$ by
Facts~\ref{para170508b}, \ref{para170508bz}, and~\ref{fact171221a}.
Also,  
we have $\edim(R)=2n+2$, and
$R$ is the completion of a local ring that is essentially of finite type over $k$, 
because of~\eqref{eq170508z} and the specific form of $I$.
In particular, $R$ is Cohen-Macaulay and complete, so it has a dualizing module.
It follows from~\cite[Corollary~4.6]{nasseh:vetfp} that $R$ has exactly two semidualizing modules, up to isomorphism, 
namely its free module of rank 1 and its dualizing module.

From the presentation~\eqref{eq170508z}, it is straightforward to show that the ideal $\p=(Z,\ulx)R$ is prime,
since $I$ is generated by monomials in $\ulx$.
Furthermore, localization at $\p$ inverts the variable $Y$,
so we have
\begin{align}
R_{\p}
&\cong \frac{k[\![Z,\underline X,Y]\!]_{(Z,\ulx)}}{(I,Z\ulx,ZY)}
= \frac{k[\![Z,\underline X,Y]\!]_{(Z,\ulx)}}{(I,Z\ulx,Z)}
\cong \frac{k[\![\underline X,Y]\!]_{(\ulx)}}{(I)}\notag\\
R_{\p}
&\cong \frac{k[\![\underline X,Y]\!]_{(\ulx)}}{(X_{1,1},X_{2,1})^2+\cdots+(X_{1,n},X_{2,n})^2}.
\label{eq180128a}
\end{align}
As 
$R_{\p}$ is a  flat, local extension of the ring
$k[\![\underline X]\!]/I\cong S_0$, and
the maximal ideal of $S_0$ extends to the maximal ideal of $R_{\p}$,
we have $r(R_{\p})=r(S_0)=2^n$ and $e(R_{\p})=e(S_0)=3^n$.
From~\eqref{eq180128a}, we have $\edim(R_\p)=2n$.

Thus, it remains to show that $R_{\p}$ has exactly $2^n$ non-isomorphic semidualizing modules. 
To this end, note that Fact~\ref{lem171124a} shows that
$$R_{\p}\cong\frac{\pseries{k}{Y,\ulx}_{(\ulx)}}{(X_{1,1},X_{2,1})^2+\cdots+(X_{1,n},X_{2,n})^2}
\cong\frac{\psx[k(\!(Y)\!)]}{(X_{1,1},X_{2,1})^2+\cdots+(X_{1,n},X_{2,n})^2}.$$
Using the right-hand quotient in this display, we conclude from~\cite[Corollary~4.11]{sather:divisor} that
$R_\p$ has exactly $2^n$  non-isomorphic semidualizing modules, as desired.\qed
\end{Proof}

\begin{disc}\label{disc170508a}
The special case $n=2$ is already interesting here.
In this case note that $r(R)=5$, which (being prime) automatically implies that $R$ has at most two semidualizing modules
by~\cite[Theorem~C(b)]{sather:bnsc}, and $e(R)=10$.
Furthermore, the localization $R_\p$ has type 4, 
embedding codepth 4,
and multiplicity 9
(the smallest 
values for these invariants
that allow for a non-trivial
semidualizing module by \emph{loc.\ cit.}, \cite{avramov:phcnr}, and Proposition~\ref{prop180113a} below, respectively).

Note that rings of composite (i.e., non-prime) type may have only trivial semidualizing modules. 
So one can imagine 
the existence of
an example of a complete Cohen-Macaulay local ring $A$ of type 4,
embedding codepth 4,
and multiplicity 9
with only trivial semidualizing modules and with a prime ideal $P$ such that
$A_P$ has non-trivial semidualizing modules. 
(It would follow automatically that $A_P$ also has type 4,
embedding codepth 4, and
multiplicity 9 
by~\cite[Theorem~C(b)]{sather:bnsc},~\cite[Corollary~3.3.12]{bruns:cmr},
Proposition~\ref{prop180113a} 
and~\cite[(40.1) Theorem]{nagata:lr}.

Alas, at this point we do not know whether such an example can actually be constructed, hence the following.
\end{disc}

\begin{question}\label{q171118a}
Let $R$ be  
Cohen-Macaulay with $r(R)=4=\ecodepth(R)$, 
$e(R)=9$ and only trivial semidualizing modules. 
For each non-maximal prime ideal $\p\in\spec(R)$, can the localization $R_{\p}$  have non-trivial semidualizing modules?
\end{question}

Noting that the ring in our example is not an integral domain, we also pose the following.

\begin{question}\label{q171118axxx}
Let $R$ be a Cohen-Macaulay local normal domain with  only trivial semidualizing modules. 
For each non-maximal prime ideal $\p\in\spec(R)$, can the localization $R_{\p}$  have non-trivial semidualizing modules?
(Note that such an example would necessarily have $\dim(R)\geq 4$ since it would satisfy $(R_1)$.)
\end{question}

Here is the proposition referred to in Remark~\ref{disc170508a}.
In the case 
where $R$ is generically Gorenstein with $\dim(R)=1$ and
$e(R)\leq 6$, it can be obtained from~\cite[Theorem~1.4(2)]{MR3272067} 
and~\cite[Proposition~3.1]{sather:divisor}.

\begin{prop}\label{prop180113a}
Let $R$ be a Cohen-Macaulay local ring with $e(R)\leq 8$.
Then $R$ has 
at most two non-isomorphic semidualizing modules, namely its free module of rank 1 and its dualizing module (if it has one).
\end{prop}

\begin{proof}
If $r(R)\leq 3$, then the result follows from~\cite[Theorem~C(b)]{sather:bnsc}.
Thus, we assume without loss of generality that $r(R)\geq 4$.
If $\m_R$ is principal, then $R$ is a hypersurface, hence Gorenstein, and we are done.
Thus, we assume without loss of generality that $\edim(R)\geq 2$.

A result of Grothendieck~\cite[Proposition (0.10.3.1)]{grothendieck:ega3-1} provides
a flat local ring homomorphism $(R,\m_R,k)\to(R',\m_{R'},k')$ such that
$k'$ is algebraically closed and $\m_{R'}=\m_R R'$. Composing with the natural map from $R'$ to its
$\m'$-adic completion, we assume that $R'$ is complete.
Then $R'$ is Cohen-Macaulay  with $e(R')=e(R)\leq 8$.
By~\cite[Theorem II(c)]{frankild:rrhffd},   it suffices to prove the result for $R'$, so we  assume
that $R$ is complete with algebraically closed (hence, infinite) residue field.

Thus, the maximal ideal $\m_R$ has a minimal reduction $(\x)R$ generated by part of a minimal generating sequence for $\m_R$.
In particular, the sequence $\x$ is a system 
of
parameters for $R$, so it is $R$-regular since $R$ is Cohen-Macaulay. 
The quotient $R''=R/(\x)R$ is artinian of length equal to $e(R)\leq 8$.
Another application of \emph{loc. cit.} shows that 
it suffices to prove the result for $R''$, so we  assume
that $R$ is artinian with length at most 8.

Claim: If $\soc(R)\cap(\m_R\ssm\m_R^2)\neq\emptyset$, 
then $\m_R$ is decomposable. 
Indeed, if $z\in \soc(R)\cap(\m_R\ssm\m_R^2)$, then 
the composition of the maps $k\xra\cong Rz\xra\subseteq\m_R\onto\m_R/\m_R^2\onto k$
shows that $Rz$ is a direct summand of $\m_R$. Since $\m_R$ is not principal by assumption,
it follows that  $\m_R$ is decomposable, as desired.

From the claim, if $\soc(R)\cap(\m_R\ssm\m_R^2)\neq\emptyset$, then we are done by~\cite[Corollary~4.6]{nasseh:vetfp}.
Thus, we assume without loss of generality that $\soc(R)\subseteq\m_R^2$.
It follows that
$4\leq r(R)=\len_R(\soc(R))\leq\len_R(\m_R^2)$,
so 
$$\edim(R)=\len(R)-1-\len_R(\m_R^2)\leq 8-1-4=3.$$
Thus, the desired conclusion follows from~\cite{avramov:phcnr} or~\cite[Theorem~A]{nasseh:lrec3}.
\end{proof}

\section{Fiber Products of Finite Cohen-Macaulay Type}\label{sec171118c}

The main goal of this section is to prove Theorem~\ref{thm171221a} from the introduction. 
See
also Theorem~\ref{thm171229a} for a non-Cohen-Macaulay result.

\begin{notn}
\label{notn171222a}
Let $\reg(R)$ denote the set of $R$-regular elements of the local ring $R$,
and let
$Q(R)$ be the total quotient ring of $R$.
To be clear, in the case $\depth(R)=0$, we simply have $Q(R)=R$,
and in the case $\depth(R)\geq 1$, we have $Q(R)=\reg(R)^{-1}R$. 
Note that this is the same as the usual definition of $Q(R)$ where one inverts all the non-zero-divisors of $R$.
Indeed, the depth-0 case is clear.
In the case of positive depth, since $R$ is local, the set $\operatorname{NZD}(R)$ 
of non-zero-divisors of $R$ is the union of $\reg(R)$ and the set $R^\times$ of units of $R$;
from this, it follows readily that $\operatorname{NZD}(R)^{-1}R=\reg(R)^{-1}R$.

Let $\wti R$ denote the integral closure of $R$ in $Q(R)$.
In the case $\depth(R)=0$ this simply yields $\wti R=R$.
\end{notn}

Note that the conclusion of the next result fails if $\depth(S)=0$ because then
$\wti{S\times_kT}=S\times_kT\neq 
\wti S\times\wti T$.

\begin{lem}\label{lem171223az}
Let $R\cong S\times_k T$ be a non-trivial fiber product with 
$\depth(R)\geq 1$, that is, with 
$\depth(S),\depth(T)\geq 1$.
Then there is a natural $R$-module isomorphism
$\wti{R}\cong\wti S\times\wti T$.
\end{lem}

\begin{proof}
Recall that the definition of $S\times_k T$ implies in particular that 
it
is naturally a subring of the cartesian product $S\times T$, so
$$S\times_kT\subseteq S\times T\subseteq\wti S\times\wti T\subseteq Q(S)\times Q(T).$$
Let $i\colon S\times_kT\to Q(S)\times Q(T)$ be the inclusion map.

It is straightforward to show that $\reg(S\times_kT)=\reg(S)\times\reg(T)$.
It follows that 
$$i(\reg(S\times_kT))=i(\reg(S)\times\reg(T))\subseteq Q(S)^\times\times Q(T)^\times=(Q(S)\times Q(T))^\times.$$
Using the universal mapping property for localizations, one concludes that $i$ induces a well-defined ring homomorphism
$i'\colon Q(S\times_kT)\to  Q(S)\times Q(T)$ given by
$i'((s,t)/(\sigma,\tau))=(s/\sigma,t/\tau)$.

Claim 1: the map $i'$ is an isomorphism. 
To see that $i'$ is injective, notice that $0=i'((s,t)/(\sigma,\tau))=(s/\sigma,t/\tau)$ implies for instance $s/\sigma=0$ in $Q(S)$.
Since $Q(S)$ is obtained from $S$ by inverting non-zero-divisors, it follows that $s=0$, and similarly $t=0$, so
$(s,t)/(\sigma,\tau)=(0,0)/(\sigma,\tau)=0$.

To see that $i'$ is surjective, let $(s/\sigma,t/\tau)\in Q(S)\times Q(T)$. 
Note that we cannot automatically conclude that $(s/\sigma,t/\tau)=i'((s,t)/(\sigma,\tau))\in\im(i')$
because we do not know whether $(s,t)$ is in $S\times_kT$. 
However, since $(\sigma,\tau)$ is in $\m_S\times\m_T$, we do have $(s\sigma,t\tau)\in\m_S\times\m_T\subseteq S\times_kT$.
Thus, we have
$$(s/\sigma,t/\tau)=(s\sigma/\sigma^2,t\tau/\tau^2)=i'((s\sigma,t\tau)/(\sigma^2,\tau^2))\in\im(i')$$
so $i'$ is surjective. This establishes Claim~1.

Claim 2: $i'(\wti{S\times_kT})=\wti S\times\wti T$.
To verify the containment $\subseteq$, let $(s,t)/(\sigma,\tau)\in \wti{S\times_kT}$ be given; we need to show that
$(s/\sigma,t/\tau)\in \wti S\times\wti T$.
That is, we need $s/\sigma\in \wti S$
and $t/\tau\in\wti T$.
We have already seen that $s/\sigma\in Q(S)$
and $t/\tau\in Q(T)$.
Consider a monic polynomial
$f(X)=X^n+(a_1,b_1)X^{n-1}+\cdots+(a_n,b_n)$ over $S\times_kT$ satisfied by
$(s,t)/(\sigma,\tau)$.
Set $g(X)=X^n+a_1X^{n-1}+\cdots+a_n\in S[X]$
and $h(X)=X^n+b_1X^{n-1}+\cdots+b_n\in T[X]$.
Then it suffices to show that $g(s/\sigma)=0$ 
and $h(t/\tau)=0$.
Direct computation shows that
\begin{align*}
(0,0)
&=i'(0)
=i'\left(f\left(\frac{(s,t)}{(\sigma,\tau)}\right)\right)
=\left(g\left(\frac{s}{\sigma}\right),h\left(\frac{t}{\tau}\right)\right)
\end{align*}
This establishes 
the containment $\subseteq$.

For the reverse containment $\supseteq$, we first show that 
$i'(\wti{S\times_kT})\supseteq\wti S\times 0$.
To this end, let $s/\sigma\in\wti S$ be given.
By assumption, we have $\sigma\in\m_S$, so $s\sigma\in\m_S$;
thus, because of the equality $s/\sigma=s\sigma/\sigma^2$ we may assume without loss of generality that $s\in\m_S$.
In particular, we have $(s,0)\in S\times_kT$. 

Now, let  $g(X)=X^n+a_1X^{n-1}+\cdots+a_n\in S[X]$ be a monic polynomial satisfied by
$s/\sigma$. Multiply by $X$ if necessary to assume that $a_n=0$.
For $j=1,\ldots,n-1$ let $b_j\in T$ be such that $(a_j,b_j)\in S\times_k T$, and let $b_n=0$.
Then with 
$h(X)=X^n+b_1X^{n-1}+\cdots+b_n\in T[X]$, we have $h(0)=b_n=0$.
Consider the monic polynomial
$f(X)=X^n+(a_1,b_1)X^{n-1}+\cdots+(a_n,b_n)$ over $S\times_kT$,
and let $\tau\in\reg(T)$.
Using the element $(s,0)/(\sigma,\tau)\in Q(S\times_kT)$ 
we again compute directly:
$$i'\left(f\left(\frac{(s,0)}{(\sigma,\tau)}\right)\right)=\left(g\left(\frac{s}{\sigma}\right),h\left(\frac{0}{\tau}\right)\right)=(0,0).$$
The fact that $i'$ is injective implies that $f\left((s,0)/(\sigma,\tau)\right)=0$, so
$(s,0)/(\sigma,\tau)\in\wti{S\times_kT}$.
By construction, we see that
$$\left(s/\sigma,0\right)
=\left(\frac{s}{\sigma},\frac{0}{\tau}\right)
=i'\left(\frac{(s,0)}{(\sigma,\tau)}\right)\in i'(\wti{S\times_kT}).$$
This implies that $i'(\wti{S\times_kT})\supseteq\wti S\times 0$.

The containment $i'(\wti{S\times_kT})\supseteq 0\times \wti T$ is established similarly.
Since $i'(\wti{S\times_kT})$ is closed under addition, it follows that
$i'(\wti{S\times_kT})\supseteq\wti S\times  \wti T$, as desired.
This establishes Claim~2.

In particular, Claim 2 implies that the map $i'$ induces an $R$-module isomorphism
$\wti{R}\cong\wti S\times\wti T$,
as desired.
\end{proof}

\begin{lem}\label{lem171223a}
Let $R\cong S\times_k T$ be a non-trivial fiber product with $\depth(R)\geq 1$, that is, with $\depth(S),\depth(T)\geq 1$.
Then there is a natural $R$-module isomorphism
$(\m_R\wti R+R)/R\cong((\m_S\wti S+S)/S)\oplus((\m_T\wti T+T)/T)$.
In particular, the module $(\m_R\wti R+R)/R$ is cyclic if and only if 
$(\m_S\wti S+S)/S$ and $(\m_T\wti T+T)/T$ are both cyclic and one of them is 0.
\end{lem}

\begin{proof}
We note that if $A$ is a local ring, then $\m_A\wti A\cap A=\m_A$. Indeed, 
the containment $\m_A\wti A\cap A\supseteq\m_A$ holds because $\wti A$ is a ring extension of $A$.
For the reverse containment,
use the Cohen-Seidenberg theorems to find  a maximal ideal $\n\subseteq\wti A$ such that
$\n\cap A=\m_A$. This implies that $\m_A\supseteq\m_A\wti A\cap A$,
as desired.

In particular, this implies that
\begin{align*}
\frac{\m_A\wti A+A}{A}
&\cong\frac{\m_A\wti A}{\m_A\wti A\cap A}
=\frac{\m_A\wti A}{\m_A}.
\end{align*}

Now, for the isomorphism
$(\m_R\wti R+R)/R\cong((\m_S\wti S+S)/S)\oplus((\m_T\wti T+T)/T)$,
we use the previous paragraph along with Lemma~\ref{lem171223az}:
\begin{align*}
\frac{\m_R\wti R+R}{R}
&\cong\frac{\m_R\wti R}{\m_R}
\cong\frac{(\m_S\oplus\m_T)(\wti S\times\wti T)}{\m_S\oplus\m_T}
\cong\frac{(\m_S\wti S)\oplus(\m_T\wti T)}{\m_S\oplus\m_T}\\
&\cong\frac{\m_S\wti S}{\m_S}\oplus\frac{\m_T\wti T}{\m_T}
\cong\frac{\m_S\wti S+S}{S}\oplus\frac{\m_T\wti T+T}{T}
\end{align*}
In particular, this shows that
$(\m_R\wti R+R)/R$ is cyclic if and only if 
$(\m_S\wti S+S)/S$ and $(\m_T\wti T+T)/T$ are both cyclic and 
one of them is 0.
\end{proof}

\begin{lem} \label{prop171230a}
Assume that $\m_R$ is
decomposable. 
If $R$ has finite Cohen-Macaulay type, then $\dim(R)\geq 1$.
\end{lem}

\begin{proof}
If 
$R$ has finite Cohen-Macaulay type and
$\dim(R)=0$, then $R$ is a principal ideal ring by~\cite[3.3~Theorem]{MR2919145};
however, this contradicts the fact that $\m_R$ is decomposable.
\end{proof}

\begin{disc}
Lemma~\ref{prop171230a} shows that if $(S,\m_S,k)$ and $(T,\m_T,k)$ have finite Cohen-Macaulay type,
then $R=S\times_kT$ may not have finite Cohen-Macaulay type.
Indeed, consider the rings 
$S=k[x]/(x^2)$ and $T=k[y]/(y^2)$, both of which have finite Cohen-Macaulay type
since they are principal ideal rings; see~\cite[3.3~ Theorem]{MR2919145}.
However, the fiber product
$R=S\times_kT\cong k[x,y]/(x^2,xy,y^2)$ is not a principal ideal ring,
so it has infinite (i.e., non-finite) Cohen-Macaulay type by \emph{loc.\ cit}.
\end{disc}

\begin{lem}\label{lem171229a}
Let $(R,\m_{R})$ be a local 
non-artinian
ring such that $\m_{R}^i\cap\Nil(R)=0$ for $i\gg 0$. 
Then $e(R/\Nil(R))=e(R)$.
In particular, if $R$ has dimension 1 and finite Cohen-Macaulay type, then $e(R)=e(R/\Nil(R))\leq 3$.
\end{lem}

\begin{proof}
Set $R'=R/\Nil(R)$,
and observe that the map $\spec(R)\to\spec(R')$ given by $\p\mapsto\p':=\p/\Nil(R)$ is a well-defined bijection.
Furthermore the isomorphism $R'/\p'\cong R/\p$ implies that $\dim(R')=\dim(R)$ and $\dim(R'/\p')=\dim(R/\p)$.
In particular, one has $\dim(R'/\p')=\dim(R')$ if and only if $\dim(R/\p)=\dim(R)$. 

Now, for each 
prime
$\p\in
\Min(R)$, 
we have 
$\p\neq\m_{R}$ since $R$ is not artinian, so
$$0=(\m_{R}^i\cap\Nil(R))_\p=(\m_{R})_\p^i\cap\Nil(R)_\p=R_\p\cap\p_\p=\p_\p.$$
Thus, the ring $R_\p$ is a field, so $\len(R_\p)=1$.
Since the ring $R'$ is reduced, the ring $(R')_{\p'}$ is also a field, so 
$\len(R'_{\p'})=1$ as well. 
Thus, the isomorphism $R'/\p'\cong R/\p$ conspires with 
the  
formula~\eqref{eq180119a}
to imply the following.
\begin{align*}
e(R)
=\sum_\p\len(R_\p)e(R/\p)
=\sum_{\p'}\len(R'_{\p'})e(R'/\p')
=e(R')
\end{align*}
Here the 
first
sum is taken over all $\{\p\in\Min(R)\mid\dim(R/\p)=\dim(R)\}$,
and similarly for the second sum.

In particular, if $R$ has dimension 1 and finite Cohen-Macaulay type, then $R'$ has finite Cohen-Macaulay type 
and we have $\m_{R}^i\cap\Nil(R)=0$ for $i\gg 0$
by~\cite[4.16~Theorem]{MR2919145}.
Thus, it follows from the previous paragraph that $e(R)=e(R')\leq 3$
because of~\cite[4.10~and~A.29(iv)~Theorems]{MR2919145}.
\end{proof}

\begin{lem} \label{lem171230a}
Let $R\cong S\times_k T$ be a non-trivial fiber product, and
let $M$ and $M'$ be finitely generated $S$-modules. 
\begin{enumerate}[\rm(a)]
\item \label{lem171230a1}
The module $M\oplus 0$ is indecomposable over $R$  if and only if $M$ is indecomposable over $S$.
\item \label{lem171230a2}
One has $M\oplus 0\cong M'\oplus 0$ over $R$ if and only if $M\cong M'$ over $S$.
\item \label{lem171230a3}
One has $\depth_R(M\oplus 0)=\depth_S(M)$. 
\item \label{lem171230a4}
The module $M\oplus 0$ is  maximal Cohen-Macaulay over $R$ 
if and only if $M$ is maximal Cohen-Macaulay over $S$ and $\dim(S)\geq\dim(T)$.
\end{enumerate}
\end{lem}

\begin{proof}
\eqref{lem171230a1}
It is straightforward to show that if $M$ is decomposable over $S$, then $M\oplus 0$ is decomposable over $R$.
Conversely, if $M\oplus 0$ is decomposable over $R$, then it must decompose as
$M\oplus 0\cong (M_1\oplus 0)\oplus(M_2\oplus 0)$, so $M$ must be decomposable over $S$.
(Alternately, if $M\oplus 0$ is decomposable over $R$, then it must admit an idempotent $R$-module endomorphism
$\phi\colon M\oplus 0\to M\oplus 0$ with $\phi\neq 0,1$. 
Since $M\oplus 0$ is annihilated by $\m_T$ and $R/\m_T\cong S$, it follows that $\phi$ is an 
idempotent 
$S$-module
endomorphism with $\phi\neq 0,1$
so $M$ is decomposable over $S$.)

\eqref{lem171230a2}
As in the alternate proof in the preceding paragraph, the point here is that
$\Hom_R(M\oplus 0,M'\oplus 0)$ is naturally identified with $\Hom_S(M,M')$,
and via this identification $R$-module isomorphisms on  one side correspond exactly to $S$-module isomorphisms on the other side. 

\eqref{lem171230a3}
One checks readily that a sequence $(s_1,t_1),\ldots,(s_n,t_n)\in\m_S\oplus\m_T=\m_R$ is $M\oplus 0$-regular
if and only if the sequence $s_1,\ldots,s_n\in\m_S$ is $M$-regular,
if and only if the sequence $(s_1,0),\ldots,(s_n,0)\in\m_S\oplus\m_T=\m_R$ is $M\oplus 0$-regular.
This yields inequalities $\depth_R(M\oplus 0)\leq\depth_S(M) \leq \depth_R(M\oplus 0)$, as desired.

\eqref{lem171230a4}
For the forward implication
assume that $M\oplus 0$ is  maximal Cohen-Macaulay over $R$. 
This explains the second equality in the next sequence
$$\depth_S(M)=\depth_R(M\oplus 0)=\dim(R)=
\max\{\dim(S),\dim(T)\}
\geq\dim(S).$$
The first equality is from part~\eqref{lem171230a3}, and the third one is from Fact~\ref{para170508b}.
Since the inequality $\depth_S(M)\leq\dim(S)$ is automatic, it follows that the inequality in this display is an equality,
so  $M$ is a maximal Cohen-Macaulay over $S$ and $\dim(S)\geq
\max\{\dim(S),\dim(T)\}
\geq\dim(T)$, as desired.

Conversely, assume that $M$ is maximal Cohen-Macaulay over $S$ and $\dim(S)\geq\dim(T)$.
Then as in the preceding paragraph, we have
$$\depth_R(M\oplus 0)=\depth_S(M)=\dim(S)=
\max\{\dim(S),\dim(T)\}
=\dim(R)$$
so $M\oplus 0$ is  maximal Cohen-Macaulay over $R$, as desired.
\end{proof}

The next result follows directly from Lemma~\ref{lem171230a}.

\begin{prop} \label{prop171230b}
Let $R\cong S\times_k T$ be a non-trivial fiber product.
If $\dim(S)\geq\dim(T)$
and $R$ has finite Cohen-Macaulay type, then 
$S$ has finite Cohen-Macaulay type.
\end{prop}

\begin{disc}
Remark~\ref{disc171230a} includes an example showing that, in the notation of Proposition~\ref{prop171230b},
If $\dim(S)>\dim(T)$
and $R$ has finite Cohen-Macaulay type, then 
$T$ may have infinite Cohen-Macaulay type.
\end{disc}

\begin{Proof}[\protect{Proof of Theorem~\ref{thm171221a}}]\label{para180107c}
Fact~\ref{disc161010a} implies that $R\cong S\times_k T$ is a non-trivial fiber product.

\eqref{thm171221a1}$\implies$\eqref{thm171221a3}
Assume that $R$ is Cohen-Macaulay and has finite Cohen-Macaulay type. 
In particular, we have $\dim(R)\leq 1$ by Fact~\ref{para170508b}.
Thus, we have $\dim(R)=1$ by Lemma~\ref{prop171230a}.
Since $R$ is Cohen-Macaulay, 
$S$ and $T$ must be Cohen-Macaulay of Krull dimension 1
by Fact~\ref{para170508b}, and they must have finite Cohen-Macaulay type by Proposition~\ref{prop171230b}.
Thus, 
Lemma~\ref{lem171229a} implies that $e(R)\leq 3$.

\eqref{thm171221a3}$\implies$\eqref{thm171221a2}
Assume that $R\cong S\times_k T$ where
$S$ and $T$ are both Cohen-Macaulay of Krull dimension 1 with finite Cohen-Macaulay type and $e(R)\leq 3$.
Fact~\ref{fact171221a} implies that $e(S)+e(T)=e(R)\leq 3$.
Since $e(S)$ and $e(T)$ are positive integers, we re-order $S$ and $T$ if necessary to conclude that
$e(S)\leq 2$ and $e(T)=1$. In particular, since $T$ is Cohen-Macaulay, it is unmixed, and it follows that $T$ is regular.
Using~\cite[Exercise~4.6.14(b)]{bruns:cmr}, we have
$2 \ge e(S) \ge \ecodepth(S)+1$.
This implies $\ecodepth(S) \le 1$, so $S$ is a hypersurface.
Also, $S$ is analytically unramified by~\cite[4.15~Proposition]{MR2919145}.

\eqref{thm171221a2}$\implies$\eqref{thm171221a1}
Assume that $R\cong S\times_k T$ where
$(S,\m_S,k)$ is 
an
analytically unramified 
hypersurface
of Krull dimension 1 and multiplicity at most 2
and $(T,\m_T,k)$ is regular of Krull dimension 1.
Facts~\ref{para170508b} and~\ref{fact171221a} imply that $R$ is Cohen-Macaulay of dimension 1
with $e(R)=e(S)+e(T)\leq 2+1=3$. 
Since $T$ is regular, its completion $\comp T$ is regular, so reduced.
By assumption, $\comp S$ is also reduced, hence so is $\comp R\cong\comp S\times_k\comp T$.
Thus, according to~\cite[4.10~Theorem]{MR2919145}, it remains to show that the $R$-module
$(\m_R\wti R+R)/R$ is cyclic; see the introduction of~\cite[Chapter~4]{MR2919145}.

Since $T$ is integrally closed, we have $(\m_T\wti T+T)/T=(\m_TT+T)/T=T/T=0$.
Since $S$ is Cohen-Macaulay and analytically unramified of dimension 1 and multiplicity 2, we know that
$S$ has finite Cohen-Macaulay type by~\cite[4.18~Theorem]{MR2919145},
hence $(\m_S\wti S+S)/S$ is cyclic by~\cite[4.10~Theorem]{MR2919145}.
It follows from Lemma~\ref{lem171223a} that $(\m_R\wti R+R)/R$ is cyclic, as desired.\qed
\end{Proof}

\begin{disc}
Note that the analytically unramified condition on $S$ in condition~\eqref{thm171221a2}
of Theorem~\ref{thm171221a} is necessary. 
Indeed, consider for example the ring $R=k[\![X,Y]\!]/(X^2)\times_kk[\![Z]\!]\cong k[\![X,Y,Z]\!]/(X^2,XZ,YZ)$. Then
$R$ is not reduced, so~\cite[4.15~Proposition]{MR2919145}
implies that 
$R$ has infinite Cohen-Macaulay type.
\end{disc}

\begin{cor}\label{cor171229a}
Assume that $R$ is complete and equicharacteristic 
with decomposable maximal ideal,
and that $k$ is an algebraically closed field of characteristic $0$.
Then the following conditions are equivalent.
\begin{enumerate}[\rm(i)]
\item\label{cor171229a1}
$R$ is Cohen-Macaulay  of finite 
Cohen-Macaulay
type.
\item\label{cor171229a2}
$R$ is isomorphic to one of the following.
\begin{enumerate}[\rm(1)]
\item
$k[\![x]\!]\times_k k[\![z]\!]\cong k[\![x,z]\!]/(xz)$,
\item
$k[\![x,y]\!]/(x^2-y^n)\times_k k[\![z]\!]\cong k[\![x,y,z]\!]/(x^2-y^n,xz,yz)$ for some $n\ge2$.
\end{enumerate}
\end{enumerate}
\end{cor}

\begin{proof}
\eqref{cor171229a2}$\implies$\eqref{cor171229a1}
It is straightforward to show that the ring $k[\![x,y]\!]/(x^2-y^n)$ is complete and reduced with multiplicity 2
when
$n\geq 2$; in the case $n=2$, this uses the fact that the characteristic of $k$ is not 2.
In particular,
this ring is
analytically unramified. 
Since 
$k[\![x]\!]$ is also analytically unramified,
this implication follows from Theorem~\ref{thm171221a}.

\eqref{cor171229a1}$\implies$\eqref{cor171229a2}
Assume that $R$ is Cohen-Macaulay  of finite 
Cohen-Macaulay
type.
Since $R$ has decomposable maximal ideal by assumption,
Theorem~\ref{thm171221a} implies that 
$R\cong S\times_k T$ where
$S$ is an analytically unramified hypersurface of Krull dimension 1 and multiplicity at most 2,
and $T$ is regular of Krull dimension 1.
Since $S$ and $T$ are homomorphic images of the complete, equicharacteristic local ring $R$,
the rings $S$ and $T$ are also complete and equicharacteristic. 
Cohen's structure theorem implies that $T\cong k[\![z]\!]$.

If $e(S)=1$, then $S$ must be regular, since it is unmixed; in this case, we have $S\cong k[\![x]\!]$, as desired. 
So, assume that $e(S)=2$. 
Since $S$ is a complete hypersurface of dimension 1,
we can write
$S\cong k[\![x,y]\!]/(f)$
for some $f\in k[\![x,y]\!]$.
Applying the Weierstrass preparation theorem, we may assume that $f=x^e-y^n$ for some $0<e\le n$; see the proof 
of~\cite[Theorem (8.8)]{MR1079937}.
It follows that $e=e(S)=2$, so $S$ has the desired form.
\end{proof}

Next, we move to the non-Cohen-Macaulay situation.

\begin{thm}\label{thm171229a}
Assume that $R$ has decomposable maximal ideal and
$\dim(R)\leq 1$.
Then the following conditions are equivalent.
\begin{enumerate}[\rm(i)]
\item\label{thm171229a1}
$R$ has finite Cohen-Macaulay type.
\item\label{thm171229a2}
$R\cong S\times_kT$ for local rings $S$ and $T$ satisfying one of the following:
\begin{enumerate}[\rm(1)]
\item $S$ has dimension 1 and finite Cohen-Macaulay type, and $\dim(T)=0$, or
\item $S$ and $T$ have dimension 1 and finite Cohen-Macaulay type such that $e(S)\leq 2$ and $e(T)=1$.
\end{enumerate}
\end{enumerate}
In particular, if $R$ has finite Cohen-Macaulay type, decomposable maximal ideal, and Krull dimension at most 1, then $\dim(R)=1$.
\end{thm}

\begin{proof}
Since the maximal ideal of $R$ is decomposable, we may write $R=S\times_kT$, where $(S,\m_S,k)$ and $(T,\m_T,k)$ are local rings with 
$S\ne k\ne T$.
By virtue of~\cite[4.16~Theorem]{MR2919145}, a local ring $A$ of 
Krull dimension 1 has finite Cohen-Macaulay type if and only if so does the ring 
$A/\Nil (A)$ and one has $\m_A^i\cap\Nil (A)=0$ for $i\gg0$.
A straightforward computation shows that $\m_R^i\cap\Nil (R)=(\m_S^i\cap\Nil (S))\oplus(\m_T^i\cap\Nil (T))$ for each $i>0$.
This    implies that $\m_R^i\cap\Nil (R)=0$ for $i\gg0$ if and only if $\m_S^i\cap\Nil (S)=0=\m_T^i\cap\Nil (T)$ for $i\gg0$.
Also, note that if a local ring $A$ has Krull dimension $1$, then $A/\Nil (A)$ is a Cohen-Macaulay local ring of Krull dimension $1$, as it is reduced.
We use these facts frequently in the following.

\eqref{thm171229a1}$\implies$\eqref{thm171229a2}
Assume that $R$ has finite Cohen-Macaulay type,
and recall the assumption $\dim(R)\leq 1$. 
Lemma~\ref{prop171230a} then implies that $\dim(R)=1$.
By Fact~\ref{para170508b}  we may re-order $S$ and $T$ if necessary
to assume that one of the following holds:
\begin{enumerate}[\quad(I)]
\item\label{item171230a}
$\dim S=1$ and $\dim T=0$, or
\item\label{item171230b}
$\dim S=1=\dim T$.
\end{enumerate}
In the first of these cases, Proposition~\ref{prop171230b} implies that $S$ has finite Cohen-Macaulay type,
as desired.
In the second case, Proposition~\ref{prop171230b} implies that both $S$ and $T$ have finite Cohen-Macaulay type.
Fact~\ref{fact171221a}
and Lemma~\ref{lem171229a} 
imply that $e(S)+e(T)=e(R)\leq 3$; thus, in this case we 
may re-order $S$ and $T$ if necessary
to conclude that $e(S)\leq 2$ and $e(T)=1$, as desired.

\eqref{thm171229a2}$\implies$\eqref{thm171229a1}
Assume now that $S$ and $T$ satisfy condition~(1) or~(2) from the statement of the theorem.

Case~(1): $S$ has dimension 1 and finite Cohen-Macaulay type, and $\dim(T)=0$.
Our assumptions on $S$ here imply that $S/\Nil (S)$
has finite Cohen-Macaulay type and $\m_S^i\cap\Nil (S)=0$ for $i\gg0$.
It follows that the ring 
$$R/\Nil (R)\cong(S/\Nil (S))\times_k(T/\Nil (T))\cong S/\Nil (S)$$
has finite Cohen-Macaulay type.
Since $\m_T^i=0$ for $i\gg 0$, we also have $\m_R^i\cap\Nil (R)=0$ for $i\gg0$, so $R$ has finite Cohen-Macaulay type.

Case~(2): $S$ and $T$ have dimension 1 and finite Cohen-Macaulay type such that $e(S)\leq 2$ and $e(T)=1$.
It follows that $S/\Nil (S)$ and $T/\Nil (T)$
are both Cohen-Macaulay of Krull dimension 1 with finite Cohen-Macaulay type. 
Lemma~\ref{lem171229a} 
implies that $e(S/\Nil(S))=e(S)\leq 2$ and $e(T/\Nil(T))=e(T)=1$,
so Fact~\ref{fact171221a} says
$$e(R/\Nil (R))=e((S/\Nil (S))\times_k(T/\Nil (T)))=e(S/\Nil (S))+e(T/\Nil (T))\leq 3
$$
Theorem~\ref{thm171221a} implies
that $R/\Nil (R)\cong(S/\Nil (S))\times_k(T/\Nil (T))$ has finite Cohen-Macaulay type. 
In addition, we have $\m_S^i\cap\Nil (S)=0=\m_T^i\cap\Nil (T)$ for $i\gg0$,
so $\m_R^i\cap\Nil (R)=0$ for $i\gg0$.
We conclude that $R$ has finite Cohen-Macaulay type, as desired. 
\end{proof}

\begin{disc}\label{disc171230a}
In Theorem~\ref{thm171229a} note that if $R\cong S\times_kT$ has finite Cohen-Macaulay type,
then it does not necessarily follow that $S$ and $T$ both have finite Cohen-Macaulay type. 
Indeed if $S$ has dimension 1 and finite Cohen-Macaulay type, and $\dim(T)=0$,
then $R$ has finite Cohen-Macaulay type by Theorem~\ref{thm171229a}.
However, in this case, $T$ has finite Cohen-Macaulay type if and only if $T$ is a principal ideal ring
(that is, if and only if $\edim(T)=1$) by~\cite[3.3~Theorem]{MR2919145}.
So, for example, if $T=k[\![x,y]\!]/(x,y)^2$, then $R$ has finite Cohen-Macaulay type, but $T$ has not.
See Proposition~\ref{prop171230b} for more about this sort of behavior.

It is also worth noting that in case~\eqref{thm171229a2}(2) of Theorem~\ref{thm171229a}, $S$ and $T$ may not be Cohen-Macaulay;
in particular, $S$ may not be a hypersurface and $T$ may not be regular, in spite of their small multiplicities.
Indeed, the  ring $T=k[\![x,y]\!]/(x^2,xy)=k[\![x,y]\!]/((x)\cap(x,y)^2)$ has $\dim(T)=1=e(T)$ and finite Cohen-Macaulay type;
see~\cite[p.~53]{MR2919145}.
A similar argument as in \emph{loc.\ cit.} shows that the ring
$S=k[\![x,y,z]\!]/(x^2,xy,xz,yz)=k[\![x,y,z]\!]/((x,y)\cap(x,z)\cap(x,y,z)^2)$
has $\dim(S)=1$ and $e(S)=2$ and finite Cohen-Macaulay type.
\end{disc}

\section*{Acknowledgments} 
We are grateful to 
Mohsen Asgharzadeh, Ananthnarayan Hariharan, and Tony Se 
for helpful discussions about this work.

\providecommand{\bysame}{\leavevmode\hbox to3em{\hrulefill}\thinspace}
\providecommand{\MR}{\relax\ifhmode\unskip\space\fi MR }
\providecommand{\MRhref}[2]{%
  \href{http://www.ams.org/mathscinet-getitem?mr=#1}{#2}
}
\providecommand{\href}[2]{#2}

\end{document}